\documentclass{article}

\usepackage{amsmath, amsfonts, amsthm, amssymb}
\usepackage{graphicx}
\usepackage{enumitem}
\usepackage{thmtools, mathtools}
\usepackage{hyperref}
\usepackage{fullpage}
\usepackage{authblk}
\usepackage{xcolor}

\newtheorem{theorem}{Theorem}
\newtheorem{lemma}[theorem]{Lemma}
\newtheorem{corollary}[theorem]{Corollary}

\newtheorem{remark}[theorem]{Remark}
\newtheorem{conjecture}[theorem]{Conjecture}

\DeclareEmphSequence{\bfseries\itshape}

\newcommand{\ra}{\rightarrow}
\newcommand{\Ra}{\Rightarrow}
\newcommand{\la}{\leftarrow}

\newcommand{\sm}{\setminus}

\newcommand{\dic}{\vec{\chi}}

\newcommand{\dichia}{\vec{\chi}_{\rm a}}

\renewcommand{\epsilon}{\varepsilon}
\renewcommand{\emptyset}{\varnothing}

\renewcommand{\phi}{\varphi}

\let\le\leqslant
\let\ge\geqslant
\let\leq\leqslant
\let\geq\geqslant

\newcommand{\first}{\mathrm{f}}
\newcommand{\last}{\mathrm{\ell}}

\newtheorem{claim}{Claim}[theorem]
\newenvironment{proofclaim}[1][]{\par\noindent {\it Proof of claim}. }{ \hfill$\lozenge$\par\addvspace{6pt plus 6pt}}

\def\bm{b_{\min}}
\def\bM{b_{\max}}    

\def\Apm{A^{+-}}
\def\Gpm{G^{+-}}

\title{Acyclic Dichromatic Number of Tournaments:\\these are the Champions}

\author[1]{Pierre Aboulker}
\author[2]{Pierre Charbit}
\author[2]{Samuel Coulomb}
\author[1]{Kathryn Nurse}
\author[3]{Lucas Picasarri-Arrieta\thanks{Research supported by JSPS KAKENHI JP20A402 and 22H05001, and by JST ASPIRE JPMJAP2302.}$^,$}

\affil[1]{DIENS, École normale supérieure, CNRS, PSL University, Paris, France.}
\affil[2]{Université Paris Cité, CNRS, IRIF, F-75006, Paris, France.}
\affil[3]{National Institute of Informatics, Tokyo, Japan.}

\begin{document}

\maketitle

\begin{abstract}
    The acyclic dichromatic number of an oriented graph is the minimum size of a vertex-partition such that the digraphs induced by any single part are acyclic, and the oriented bipartite graphs between any two parts are acyclic too. We characterize the subtournaments that must appear in every tournament with sufficiently large acyclic dichromatic number, and prove that acyclic dichromatic number satisfies a local to global property,  thereby confirming two conjectures of Bang-Jensen, Picasarri-Arrieta, and Yeo. 
\end{abstract}

\section{Introduction}
An acyclic colouring of a graph is a proper vertex colouring such that every pair of colour classes induces an acyclic graph, that is, a forest. This notion was introduced in 1973 by Gr\"unbaum \cite{Grunbaum} who proved that every planar graph has an acyclic 9-colouring, and conjectured that 5 colours are enough. This bound was successively improved by Mitchem \cite{mitchem}, Albertson and Berman \cite{albertson}, and Kostochka \cite{kostochka}, until Borodin settled this by showing that every planar graph has an acyclic 5-colouring \cite{Borodin}. In particular, this results implies that every planar graph can be decomposed into an independent set and two induced forests.

In the setting of directed graphs, vertex colouring has been widely studied since the introduction of the dichromatic number by Erd\H{o}s~\cite{erdosPNCN1979} and Neumann-Lara~\cite{VNL}. A notable achievement is the characterization of heroes -- the subtournaments that appear in all tournaments with large enough dichromatic number -- by Berger, Choromanski, Chudnovsky, Fox, Loebl, Scott, Seymour, and Thomassé \cite{hero}. Another important result is the so-called "local-to-global" property for dichromatic number proven by Harutyunyan, Le, Thomassé and Wu \cite{HARUTYUNYAN2019166}:  there is a function $f$ such that if the out-neighborhood of every vertex in a tournament $T$ has dichromatic number at most $c$, then $T$ has dichromatic number at most $f(c)$.

Recently, Bang-Jensen, Picasarri-Arrieta, and Yeo~\cite{bj25ADN} defined a directed analogue to acyclic colouring: an acyclic dicolouring of a digraph is a dicolouring where the arcs between any two colour classes form an acyclic set. They raised the question of which are the tournaments analogous to heroes for this notion, and proposed a conjecture. They also conjectured the fact that acyclic dichromatic number also satisfies a local-to-global property. In this paper we confirm these two conjectures.

\subsection{Definitions and notations}

Given a positive integer $n$, we denote by $[n]$ the set $\{1, \dots, n\}$.
All graphs and digraphs in this paper are simple, finite, and contain no parallel or anti-parallel arcs.

Let $D$ be a digraph, $u,v \in V(D)$ two vertices, and $X,Y \subseteq V(D)$ two sets of vertices. If $uv \in E(D)$, we say that $u$ is an \emph{in-neighbour} of $v$, and that $v$ is an \emph{out-neighbour} of $u$; we also write $u \ra v$ to denote that $uv$ is an arc of $D$. Given a vertex $u \in V(D)$, we denote by $u^+$ the set of its out-neighbours, and by $u^-$ the set its of in-neighbours. 
We denote by $D[X,Y]$ the subdigraph of $D$ with vertex set $X \cup Y$ and containing every arc of $D$ having one end in $X$ and the other end in $Y$.
In particular, $D[X,X]$ is the subdigraph induced by $X$, also noted $D[X]$.
We write $X \Ra Y$ to say that for all vertices $x \in X$ and $y \in Y$ we have $x \ra y$.
Given two digraphs $D_1$ and $D_2$, we denote by $D_1 \Ra D_2$ the digraph obtained from the disjoint union of $D_1$ and $D_2$ by adding all arcs from $V(D_1)$ to $V(D_2)$. 

A \emph{tournament} is an orientation of a complete graph, that is, an oriented graph with exactly one arc between every pair of vertices. We say that a tournament is \emph{transitive} if it is acyclic, and we denote by $TT_n$ the unique transitive tournament on $n$ vertices.  
Given three tournaments $T_1, T_2, T_3$, we denote by $\Delta(T_1,T_2,T_3)$ the tournament obtained from  disjoint copies of $T_1, T_2, T_3$ by adding the remaining arcs so that $T_1 \Ra T_2 \Ra T_3 \Ra T_1$. For the sake of better readability, we may write $k$ in place of $TT_k$ in this last notation; for example, $\Delta(1,k,T)$ stands for $\Delta(TT_1,TT_k,T)$.
Given two tournaments $H$ and $T$, we say that $T$ is \emph{$H$-free} if it $T$ contains no subtournament isomorphic to $H$.

A \emph{$k$-dicolouring} of a digraph $D$ is a function  $\phi\colon V(D) \rightarrow [k]$ such that $D[\phi^{-1}(c)]$ is acyclic for every colour $c\in [k]$. The sets $\phi^{-1}(c)$ for $c\in [k]$ are called the \emph{colour classes} of $\phi$.
A digraph is \emph{$k$-dicolourable} if it admits a $k$-dicolouring.  
The \emph{dichromatic number} of a digraph $D$, denoted by $\dic(D)$, is the least integer $k$ such that $D$ is $k$-dicolourable. 
A tournament $H$ is a \emph{hero} if there exists an integer $c_H$ such that every $H$-free tournament has dichromatic number at most $c_H$.
In a seminal paper, Berger, Choromanski, Chudnovsky, Fox, Loebl, Scott, Seymour, and Thomassé~\cite{hero} characterized heroes.

\begin{theorem}[Berger {\it et al.}~\cite{hero}]\label{thm:heroes}
    A tournament $H$ is a hero if and only if
    \begin{itemize}[itemsep=0pt,topsep=2pt]
        \item $H$ has a single vertex, or
        \item $H \simeq (H_1 \Ra H_2) $, where $H_1$ and $H_2$ are heroes, or
        \item $H \simeq \Delta(1, k, H')$ or $H \simeq \Delta(H', k, 1)$, where $H'$ is a hero and $k \ge 1$ an integer.
    \end{itemize}
\end{theorem}

\subsection{Acyclic dicolouring and main result}

Let $D$ be a digraph. An \emph{acyclic $k$-dicolouring} of $D$ is a $k$-dicolouring with colour classes $V_1, \dots, V_k$ such that the subdigraph $D[V_i,V_j]$ is acyclic for all $i,j \in [k]$. The \emph{acyclic dichromatic number} $\dichia(D)$ of $D$ is the least integer $k$ such that $D$ admits an acyclic $k$-dicolouring. When $X \subseteq V(D)$ is a set of vertices and $D$ is clear from the context, we may write $\dichia(X)$ in place of $\dichia(D[X])$. A tournament $H$ is a \emph{champion} if there exists an integer $c_H$ such that every $H$-free tournament has acyclic dichromatic number at most $c_H$. The acyclic dichromatic number upper-bounds the dichromatic number, so every champion is a hero, but the converse is not true. The following conjecture was posed in~\cite{bj25ADN}.

\begin{conjecture}[Bang-Jensen, Picasarri-Arrieta, and Yeo \cite{bj25ADN}]
    \label{conj:champion}
    A tournament $H$ is a champion if and only if $H$ is isomorphic to a subtournament of $TT_k \Ra (\Delta(1,1,k) \Ra TT_k)$ for some integer $k \ge 1$.
\end{conjecture}

In support of their conjecture, Bang-Jensen {\it et al.}~proved that the forward implication holds ({\it i.e.}~every champion has this form), and that the tournaments $\Delta(1,1,k)$ for $k \ge 1$, and $\Delta(1,1,1) \Ra TT_1$ are champions. We  confirm their conjecture.

\begin{restatable}{theorem}{main}\label{thm:main}
    For every integer $k \ge 1$, the tournament $TT_k \Rightarrow (\Delta(1,1,k) \Ra TT_k)$ is a champion.
\end{restatable}

 The key notion to prove this theorem is the structure called dimatching. A \emph{dimatching} is a set of pairwise disjoint arcs $\{a_1b_1, \dots, a_kb_k\}$ such that for each $i,j \in [k]$, we have $a_i \ra b_j$ when $i=j$, and $a_i \la b_j$ when $i \ne j$. In this case, we say that $a_i$ and $b_i$ are \emph{matched} together, and given two vertex sets $A$ and $B$, we say that this dimatching goes from $A$ to $B$ if $a_i \in A$ and $b_i \in B$ for all $i \in [k]$. 

Bang-Jensen {\it et al.}~\cite{bj25ADN} introduced this structure to obtain 2-dicolourable tournaments with large acyclic dichromatic number. They show that any oriented graph containing a dimatching of size $k$ has acyclic dichromatic number at least $\sqrt k$. We will prove (Lemma \ref{lem:matching_to_champion}) that if a tournament contains a sufficiently large dimatching, it contains any champion. Thus Theorem \ref{thm:main} is a consequence of the following result, which we will see follows itself easily from a previous result of Atminas~\cite{ATMINAS2022113089} on bipartite graphs.

\begin{restatable}{theorem}{dimatching}\label{thm:dimatching}
    There exists a function $f \colon \mathbb N \ra \mathbb N$ such that for all integers $k \ge 1$, every tournament with acyclic dichromatic number at least $f(k)$ contains a dimatching of size $k$.
\end{restatable}

Using this result, it is easy to prove another conjecture from Bang-Jensen \textit{et al.}~\cite{bj25ADN} which states that the acyclic dichromatic number satisfies a local-to-global property in tournaments, analogous to the result of Harutyunyan et al. for dichromatic number of tournaments \cite{HARUTYUNYAN2019166}.

\begin{restatable}{theorem}{localtoglobal}\label{thm:localtoglobal}
    There exists a function $g : \mathbb N \ra \mathbb N$ such that every tournament $T$ satisfies
    $$
        \dichia(T) \le \max_{v \in V(T)} g \big( \dichia(v^+) \big).
    $$
\end{restatable}


\section{Proofs}
This section contains three subsections,  devoted respectively to the proofs of Theorem~\ref{thm:main}, Theorem~\ref{thm:dimatching} and Theorem ~\ref{thm:localtoglobal}.
\begin{remark}
    We include in the appendix a second proof of Theorem~\ref{thm:main} which is not based on Theorem~\ref{thm:dimatching}, and thus not on the result of Atminas either. This alternative proof yields better upper bounds, and we believe the ideas developed there are of independent interest for future works on the acyclic dichromatic number.
\end{remark}

\subsection{Proof of Theorem \ref{thm:main}} 
\main*

As stated in the introduction, to deduce Theorem \ref{thm:main} from Theorem \ref{thm:dimatching}, it suffices to prove that one can find any champion in every tournament that contains a sufficiently large dimatching. For this, we make use of the following analogue of Ramsey's theorem for tournaments.

\begin{theorem}[Erd\H{o}s and Moser~\cite{Ramsey}]
    \label{thm:moser}
    For all integers $n \ge 0$, every tournament of order $2^{n}$ contains a transitive tournament of order $n+1$.
\end{theorem}

\begin{lemma}\label{lem:matching_to_champion}
    Let $T$ be a tournament and $k \ge 1$ be an integer. If $T$ contains a dimatching of size $2^{2k}+2^k+4k$, then T contains a subtournament isomorphic to $TT_k \Ra (\Delta(1,1,k) \Ra TT_k)$.
\end{lemma}
\begin{proof}
    Let $M$ be a dimatching of size $2^{2k}+2^k+4k$ in $T$. Denote $A$ and $B$ the vertex sets of size $2^{2k}+2^k+4k$ such that $M$ goes from $A$ to $B$. Let $B' \subset B$ denote the vertices of $B$ with at least $2^{k-1}+2k$ in-neighbours in $B$, and let $A'$ be the vertices of $A$ matched with vertices in $B'$. 
    It is well known that, for every integer $\ell\geq 1$, every tournament contains at most $2\ell-1$ vertices of in-degree less than $\ell$  (see~\cite[Prop.~2.2.2]{bang2018}). In particular, for $\ell=2^{k-1}+2k$, we obtain that $|A'| = |B'| \geq 2^{2k}$.

    By Theorem~\ref{thm:moser},  there exists a subset $A'' \subset A'$ of size $2k+1$ that induces a transitive tournament in $T$. Let $a \in A''$ be the ``middle" vertex in this transitive tournament (that is, the unique vertex with $k$ in-neighbours and $k$ out-neighbours in $A''$) and let $b$ be the vertex matched to $a$. By definition of $B'$, the vertex $b$ has at least $2^{k-1}+2k$ in-neighbours in $B$. At least $2^{k-1}$ of them are not matched to any vertex in $A''$, so by Theorem~\ref{thm:moser}, there exists a set $B''\subseteq B$ consisting of $k$ in-neighbours of $b$ not matched to any vertex in $A''$, and that induces a transitive tournament in $T$. It is easy to see that $A'' \cup B'' \cup \{b\}$ induces a tournament isomorphic to $TT_k \Ra (\Delta(1,1,k) \Ra TT_k)$.
\end{proof}

\subsection{Proof of Theorem \ref{thm:dimatching}}
\dimatching*
One of the first observations (that was implicit in the work of Bang-Jensen {\it et al.}) is that in order to prove a bound on the acyclic chromatic number of a tournament, it is enough to prove it in the case where the vertex set can be partitioned into two transitive tournaments (in other words, when the tournament has dichromatic number $2$). We prove below two results that are stronger than this sole fact because we feel they might be of independent interest for future work on the acyclic dichromatic number. 

\begin{lemma}\label{lem:remove_induced}
    Let $D$ be a digraph and $H$ be an induced subdigraph of $D$, then
    $$\dichia(D) \leq \dichia \big( D \setminus E(H) \big) \cdot \big( \dichia(H)+1 \big).$$
\end{lemma}
\begin{proof}
    Let $\phi_D$ and $\phi_H$ be any acyclic dicolourings of $D\setminus E(H)$ and $H$, respectively. 
    We show that the colouring $\phi^\star$ of $D$ defined below is an acyclic dicolouring of $D$, hence implying the result.
    \[
        \phi^\star \colon v \mapsto \left\{
            \begin{array}{ll}
                \big(\phi_D(v), \phi_H(v)\big) & \mbox{if $v\in V(H)$,} \\[0.1cm]
                \big(\phi_D(v), 0\big) & \mbox{otherwise.}
            \end{array}
        \right.   
    \]
    
    Suppose for contradiction that $D$, coloured with $\phi^\star$, contains a directed cycle $C$ which is monochromatic or bichromatic with alternating between two colours. Then the same holds with respect to $\phi_D$. By the choice of $\phi_D$, the cycle $C$ uses at least one arc $uv$ of $H$. Since $C$ is monochromatic or alternating between two colours, every vertex $w$ of $C$ has the same colour as $u$ or as $v$. Hence, every vertex $w$ of $C$ satisfies $\phi_H(w) \ne 0$, and since $H$ is an induced subdigraph of $D$, the cycle $C$ is in $H$. But then $C$ is monochromatic or  alternating  between two colours with regard to $\phi_H$, a contradiction.    
\end{proof}

By iterating the lemma above, we obtain the following useful corollary. Informally, it states that if the edges of a digraph $D$ can be covered (up to an acyclic set) by a bounded number of induced subdigraphs with bounded acyclic dichromatic number, then the acyclic dichromatic number of $D$ is also bounded.

\begin{corollary}\label{cor:cover}
    Let $D$ be a digraph and $D_1,\dots,D_\ell$ be a collection of induced subdigraphs of $D$ such that $D \setminus \bigcup_{i=1}^\ell A(D_i)$ is acyclic, then 
    \[
        \dichia(D) \leq \prod_{i=1}^\ell \Big( \dichia(D_i)+1 \Big).
    \]
\end{corollary}

    Let now $H_k$ be the family of tournaments defined recursively by $H_0=K_1$,  and $H_k=\Delta(1,1,H_{k-1})$ for $k\geq 1$. It is easy to observe that $H_k$ contains a dimatching of size $k$, and  by Theorem \ref{thm:heroes} it is a hero. Hence there exists an integer $h(k)$ such that if a tournament $T$ does not contain a dimatching of size $k$, then $\dic (T) \le h(k)$, so we partition $V(T)$ into $c \le h(k)$ sets $T_1,\dots,T_c$, each inducing a transitive tournament on $T$. Remark that the subtournaments induced by two colour classes $\big\{T[V(T_i \cup T_j)]\big\}_{i,j \in [c]}$ together cover all arcs of $T$. Thus, by Corollary~\ref{cor:cover},
    \[
        \dichia(T) \le \prod_{i,j\in[c]} \big( \dichia(T_i \cup T_j)+1 \big).
    \]    
    Therefore, it suffices to bound the acyclic dichromatic number of $T[T_i \cup T_j]$ for each $i,j \in [c]$. Fix $i,j \in [c]$. Note that $T[T_i]$ and $T[T_j]$ are acyclic induced subtournaments of $T[T_i \cup T_j]$, and after removing their edges, we are left with the bipartite tournament $T[T_i,T_j]$, so by Corollary~\ref{cor:cover}, we have $\dichia(T_i \cup T_j) \le 4 \cdot \dichia \big( T[T_i,T_j] \big)$. Denote $A = T_i$, $B = T_j$, and $T' = T[T_i, T_j]$.

So we consider a digraph $D$ that is the orientation of a complete bipartite graph with sides $A$ and $B$. From $D$, let us define an undirected bipartite graph $G$ on the same vertex set with an edge between $a$ and $b$ if and only if $ab$ is an arc of $D$. Observe that a dimatching from $A$ to $B$ is then simply an induced matching in $G$, and a dimatching from $B$ to $A$ is an induced co-matching (the bipartite graph obtained by complementing the edges of a matching). Also, up to doubling the number of colours, note that we can assume that in an acyclic dicolouring of $D$, every colour class is either a subset of $A$ or a subset of $B$. It is easy to see that such a colouring is valid if an only if there is no directed $C_4$ between any two colour classes. This is equivalent to say that, between any two colour classes, the graph $G$ is $2K_2$-free, that is, it does not contain any induced matching of size $2$.

Using these observations, Theorem \ref{thm:dimatching} is a direct consequence of the following result of Atminas~\cite[Theorem~2.1]{ATMINAS2022113089}.
    
\begin{theorem}[Atminas~\cite{ATMINAS2022113089}]
For all integers $m$ and $n$, there exists an integer $f(n,m)$ such that for any bipartite $G = (A \cup B, E \subset A \times B)$:
\begin{itemize}[itemsep=0pt,topsep=2pt]
    \item either $G$ contains an induced matching of size $n$,
    \item or $G$ contains an induced co-matching of size $m$,
    \item or there are two partitions $A = A_1 \cup A_2 \cup \ldots \cup A_k$ and $B= B_1 \cup B_2 \cup \ldots \cup B_\ell$ with
$k,\ell \le f (n, m)$ such that $G[A_i, B_j]$ is $2K_2$-free for all $1 \le i \le k$ and $1 \le j \le \ell$.
\end{itemize}
\end{theorem}


\subsection{Proof of Theorem \ref{thm:localtoglobal}}
\localtoglobal*
    Let $f$ be a function that verifies Theorem \ref{thm:dimatching} and $T$ a tournament with acyclic dichromatic number at least $f(2k^2)$. By Theorem \ref{thm:dimatching}, $T$ contains a dimatching $\{a_ib_i\}_{i \in [2k^2]}$ of size $2k^2$. Denote $A,B \subset V(T)$ respectively the set of tails $\{a_i\}_{i \in [2k^2]}$ and the set of heads $\{b_i\}_{i \in [2k^2]}$ of the dimatching. Note that, for every $b_i \in B$, all vertices in $A$ are out-neighbours of $b_i$ except its match $a_i$.
    It is well-known that every tournament of order $2n$ contains a vertex with out-degree at least $n$ (see~\cite[Proposition~2.2.2]{bang2018}). In particular, there exists a vertex $v \in B$ with $k^2$ out-neighbours in $T[B]$. These $k^2$ vertices and their matches in $A$ induce a dimatching of size $k^2$ in $v^+$, which by \cite[Proposition~1]{bj25ADN} has acyclic dichromatic number $k$.

\bibliographystyle{plain}
\bibliography{biblio}

\appendix
\section{Appendix}

As mentioned in the introduction, we include here a second proof of Theorem \ref{thm:main}, which is not based on Theorem \ref{thm:dimatching}.

A graph $G = (V,E)$ is an \emph{interval graph} if there exists a family $(I_v)_{v \in V}$
of intervals of the real line such that, for every $u,v \in V$, we have $uv \in E$
if and only if $I_u \cap I_v \neq \emptyset$.
We make use of the following technical lemma about interval graphs. It captures the fact that, in a connected interval graph, starting a Breadth-First Search from the interval with smallest right endpoint yields layers that are naturally ordered along the real line.

\begin{lemma}
    \label{lem:interval_graphs}
    Let $G$ be an interval graph, where for each vertex $v\in V(G)$, the interval associated to $v$ is denoted by $I_v$. There exists a partition $L_0,\dots,L_h$ of $V$ such that:
    \begin{itemize}[itemsep=0pt, topsep=2pt]
        \item for each $0 \le i \le h$, either $|L_i|=1$, or $i \ge 1$ and there exists $u\in L_{i-1}$ such that $L_i \subseteq N_G(u)$; and 
        \item for all $0\leq i\leq j-2\leq h-2$, we have $\displaystyle \max \Big( \bigcup_{v\in L_i} I_v \Big) < \min \Big( \bigcup_{v\in L_j} I_v \Big)$.
    \end{itemize}
\end{lemma}
\begin{proof}
    Let $\prec$ denote a total order on $V(G)$ defined, for $u,v \in V(G)$, by $u \prec v$ if $\max(I_u) < \max(I_v)$, and in case of equality between these maxima, we order $u$ and $v$ arbitrarily. It is classical and easy to observe that this order satisfies the following property for all $u,v,w \in V(G)$:
    \begin{equation}\label{intorder}
        \text{if $u \prec v \prec w$ and $uw$ is an edge of $G$, then $vw$ is an edge of $G$}.
    \end{equation}

    Let $a_0$ be the smallest vertex in $G$ (with respect to $\prec$), and as long as $a_i$ has neighbours larger than itself (with respect to $\prec$),  define $a_{i+1}$ as its largest neighbour. Let $a_d$ be the maximum such $a_i$. By~\eqref{intorder}, observe that every vertex between $a_{i-1}$ and $a_i$ is a neighbour of $a_i$, so all vertices smaller than $a_d$ belong to the connected component $C$ of $a_0$. Moreover, all vertices after $a_d$ belong to another connected component, as if a vertex after $a_d$ is adjacent to $C$, then it is adjacent to $a_d$ by \eqref{intorder}, but this contradicts the maximality of $a_d$. Observe that $a_i$ is at distance at most $i$ from $a_0$ (we prove below that the distance is exactly $i$), thus every vertex in $C$ is at distance at most $d+1$ from $a_0$. For each $0 \leq i \leq d+1$, define $L_i$ as the set of vertices at distance $i$ from $a_0$. Note that $L_0=\{a_0\}$.
    
     \begin{claim}
         For every $1 \le i \le d$, $a_i\in L_i$, $L_{i}\subseteq N_G(a_{i-1})$, $L_i\preceq a_{i}$ and if $i\geq 2$ then $a_{i-2}\prec L_i$.
    \end{claim}
    \begin{proofclaim}
        We proceed by induction on $i$. It is trivial for $i=1$ by definition of $a_1$ and $L_1$.  Assume now that $i\geq 2$ and that it is true for $j<i$. Assume first, for a contradiction, that $a_i \in L_j$ for $j<i$. By induction hypothesis $L_{j}\subseteq  N_G(a_{j-1})$, which contradicts the choice of $a_j$ since $a_j\prec a_{i}$. Since $a_i$ is at distance at most $i$ from $a_0$, we get $a_i\in L_i$.

        Now let $a$ be an arbitrary vertex in $L_i$. Let $j\geq 1$ be such that $a_{j-1}\prec a \preceq a_{j}$. Note that such a $j$ exists as every vertex $u\in C$ satisfies $a_0\preceq u \preceq a_d$.  By~\eqref{intorder}, $a$ is a neighbour of $a_j$. 
    
        If $j\leq i-2$, by induction $a_j\in L_j$ and $a$ is at distance at most $j+1<i$ from $a_0$, a contradiction. This shows that $a_{i-2}\prec L_i$.
        If $j=i-1$ then $a$ is a neighbour of $a_{i-1}$, as desired. If $j\geq i$, let $a'\in L_{i-1}$ be such that $a'a$ is an edge. By induction $a_{i-2}a'$ is an edge so by definition of $a_{i-1}$ we have $a'\preceq a_{i-1}\prec a$. Now by~\eqref{intorder} we have $a\in N_G(a_{i-1})$ as desired. This shows $L_i \subseteq N_G(a_{i-1})$.
    Now the definition of $a_i$ implies that $L_i \prec a_i$.
    \end{proofclaim}
    
    As observed above, the connected component of $a_0$ is precisely the set of vertices between $a_0$ and $a_d$, so if there are still remaining vertices, we can start again with the successor of $a_d$, and perform the same BFS to get the successive layers, and the sequence of partitions satisfies the requirement of the Lemma.  
\end{proof}


\subsection{Proof of Theorem~\ref{thm:main}}

This section is dedicated to the proof of Theorem~\ref{thm:main}. We first reduce to the case of bipartite tournaments ({\it i.e.} orientations of complete bipartite graphs), then we make use of the following notion.

Let $T=(A,B,E)$ be a bipartite tournament. Given an ordering $\prec$ of $B$, we say that a vertex $a\in A$ \emph{switches} at a vertex $b\in B$ if $b' \ra a \ra b$ or $b' \la a \la b$, where $b'$ is the predecessor of $b$ in $\prec$. Informally, it means that the orientations of the arcs at $a$ change, along the ordering of $B$, when arriving at $b$.

The idea of the proof is that, in a bipartite tournament $T$ with large acyclic dichromatic number, either we can find a vertex in $A$ that switches many times with respect to some ordering of $B$, or $T$ contains a large dimatching. In both cases, we can conclude that $T$ contains $TT_k\Ra (\Delta(1,1,k) \Ra TT_k)$.
The proof of Theorem~\ref{thm:main} thus relies on the following key lemma.

\begin{lemma}\label{lem:main}
    There exists a function $f\colon \mathbb{N}^3 \to \mathbb{N}$ such that, for all integers $p,q,r \ge 1$ and every bipartite tournament $T=(A,B,E)$ with  $\dichia(T) \geq f(p,q,r)$, at least one of the following holds:
    \begin{enumerate}[label=(\roman*)]
        \item for every ordering of $B$, some vertex of $A$ switches at least $r$ times; or 
        \label{enum_lemma_matching:item:1}
        \item $T$ contains a dimatching of size $p$ from $A$ to $B$; or
        \label{enum_lemma_matching:item:2}
        \item $T$ contains a dimatching of size $q$ from $B$ to $A$.
        \label{enum_lemma_matching:item:3}
    \end{enumerate}
\end{lemma}

\begin{proof}
    Let $p,q,r \ge 1$ be positive integers and $T=(A,B,E)$ a bipartite tournament such that none of~\ref{enum_lemma_matching:item:1}, \ref{enum_lemma_matching:item:2}, and~\ref{enum_lemma_matching:item:3} hold. We show that the acyclic dichromatic number of $T$ is bounded by some value $f(p,q,r)$ to be determined, which implies the desired result. We proceed by induction on $p+q+r$.
    
    If $\min\{p,q,r\}=1$, then we claim that $T$ has no directed cycle. Indeed, a vertex of $A$ on a directed cycle has at least one in-neighbour and one out-neighbour in $B$, and such a vertex switches at least once with regard to every ordering of $B$. Moreover, any arc from $A$ to $B$ and any arc from $B$ to $A$ forms a dimatching of size 1. Thus, if $\min\{p,q,r\}=1$, then $T$ is acyclic, and we can set $f(p,q,r) = 2$.

    Now suppose $p,q,r>1$. Since \ref{enum_lemma_matching:item:1} does not hold, there exists an ordering $\prec$ on $B$ such that every vertex of $A$ switches strictly less than $r$ times with respect to $\prec$. In the following, when speaking about vertices in $B$, we use the words smaller and larger to denote their relation with respect to $\prec$, and use the notations $[b,b']$ (or their variants with open brackets such as $[b,b'[$) to denote intervals in $B$ with regard to $\prec$. We denote respectively $\bm$ and $\bM$ the smallest and largest element of $B$. Let us partition $A$ into four sets depending on the orientation of the arcs at $\bm$ and $\bM$:
    \begin{itemize}[itemsep=0pt]
        \item $A^{++}=\bm^+ \cap \bM^+$,
        \item $A^{+-}=\bm^+\cap \bM^-$, 
        \item $A^{-+}=\bm^-\cap \bM^+$, and
        \item $A^{--}=\bm^-\cap \bM^-$. 
    \end{itemize}
    
    In what follows, we bound the acyclic dichromatic number of the bipartite tournament induced by each of these sets together with $B$. The result then follows from Corollary~\ref{cor:cover}. Remark that if a vertex $a \in A$ never switches then it lies in no directed cycle, so removing it does not decrease the acyclic dichromatic number, thus without loss of generality, we may assume that every vertex $a \in A$ switches at least one time. The key object in our proof is the interval graph that we now define. For each vertex $a\in A$, denote respectively $\first(a)$ and $\last(a)$ the first and last switches of $a$ in $B$, and let $I_a$ be the interval 
    \[
    I_a \coloneq [\first(a), \last(a)[\ = \{b\in B: \first(a)\preceq b \prec\last(a)\}.
    \]
    
    Let $G$ be the interval graph on $A$ defined by these intervals ({\it i.e.} $G$ has vertex set $A$, and for $a,a' \in A$, there is an edge $aa'$ in $G$ if and only if $I_a$ and $I_{a'}$ intersect).
    
    \begin{claim}\label{claim:clique}
        If $K \subseteq A$ is a clique of $G$, then $\dichia \big( T[K \cup B] \big) \le f(p,q,r-1)^2$.
    \end{claim}
    \begin{proofclaim}
        Let $a \in K$ be such that $f(a)$ is maximal. Then for $a' \in K$, we have $\first(a') \preceq f(a) \prec \last(a)$, where the first inequality holds by definition of $a$, and the second because otherwise $a$ and $a'$ are not adjacent in $G$. Hence each $a' \in K$ switches at least once in $[\bm,f(a)]$ and once in $]f(a);\bM]$, so $a'$ switches less than $r-1$ times in $[\bm,f(a)]$ and in $]f(a);\bM]$. By induction hypothesis, this implies
        \[
            \dichia \big(K \cup [\bm,f(a)]\big) \le f(p,q,r-1) -1 \text{~~~and~~~}
            \dichia \big(K \cup ]f(a),\bM]\big) \le f(p,q,r-1) - 1,
        \]
        and the claim follows by applying Corollary~\ref{cor:cover}. 
    \end{proofclaim}        
    
    \begin{claim}\label{claim:stable}
        A stable set in $A^{++}$ (resp. $A^{--}$)  has size less than $p$ (resp. $q$). 
    \end{claim}    
    \begin{proofclaim}
        Suppose $S$ is a stable set in $A^{++}$. Let $s,s' \in S$ be distinct vertices. Recall that $\bm \ra s$ and $\bM \ra s$, because $s \in A^{++}$, and as $\first(s)$ is the first switch of $s$, it holds that $s \ra \first(s)$. By definition of $I_{s'}$, we have $\first(s') \in I_{s'}$, thus $\first(s') \notin I_s$ as otherwise $s$ and $s'$ are adjacent in $G$. Hence either $\first(s') \prec \first(s)$ or $\last(s) \preceq \first(s')$. If $\first(s') \prec \first(s)$, then note that $\bm \ra s$ and $s$ does not switch in $[\bm, \first(s)[$, so $\first(s') \ra s$. If $\last(s) \preceq \first(s')$, then note that $\bM \ra s$ and $s$ does not switch in $[\last(s), \bM]$, so $\first(s') \ra s$. It follows that the arcs $\{s\ \first(s)\}_{s \in S}$ form a dimatching from $A$ to $B$, so $|S| < p$ by hypothesis.
        
        The proof for $A^{--}$ works symmetrically.
    \end{proofclaim}

    \begin{claim}
        $\dichia\big(A^{++}\cup A^{--} \cup  B\big)\leq \big(f(p,q,r-1)^2+1\big)^{p+q}$
    \end{claim}
    \begin{proofclaim}
        Since any interval graph is perfect, its vertices can be covered by a number of cliques equal to its stability number. Hence by Claim~\ref{claim:stable}, we can partition $A^{++}$ and $A^{--}$ into $p$ and $q$ cliques, respectively. Then the claim follows by combining Claim~\ref{claim:clique} with Corollary~\ref{cor:cover}.
    \end{proofclaim}

    It remains to bound $\dichia(A^{+-}\cup B)$ and $\dichia(A^{-+}\cup B)$. This time, $G$ may have unbounded stability number, so we use a different strategy based on the partition guaranteed by Lemma~\ref{lem:interval_graphs}. 

    \begin{claim}\label{claim:c4}
        Let $b_1$ and $b_2$ be vertices in $B$ such that $(b_1^- \cap b_2^+) \ne \emptyset$. Then
        \[
            \dichia \Big( (b_1^+ \cap b_2^-) \cup B \Big) \leq f(p-1,q,r) \cdot f(p,q-1,r).
        \]
    \end{claim}
    \begin{proofclaim}
        Let $a \in (b_1^- \cap b_2^+)$. Observe that every dimatching from $(b_1^+\cap b_2^-)$ to  $a^-$ can be augmented by the arc $ab_1$, so by induction hypothesis, it holds that
        \[
            \dichia \big( (b_1^+\cap b_2^-) \cup a^- \big)\leq f(p-1,q,r)-1.
        \]
        Similarly, every dimatching from  $a^+$ to $(b_1^+\cap b_2^-)$ can be augmented by the arc $b_2a$, and thus 
        \[
            \dichia \big( (b_1^+\cap b_2^-) \cup a^+ \big) \leq f(p,q-1,r)-1.
        \]
        The claim follows from these two inequalities together with Corollary~\ref{cor:cover}.
    \end{proofclaim} 

    The previous claim applied to $\bm$ and $\bM$ shows that if both $A^{+-}$ and $A^{-+}$ are both nonempty, then we are done. Hence (up to reversing all the arcs, and swapping $p$ and $q$) we can assume that $A^{-+}$ is empty. It remains to bound $\dichia(A^{+-}\cup B)$.  For better readability, we denote $G[A^{+-}]$ by $G^{+-}$.
    
    Let $a\in \Apm$. By definition of $\first(a)$ and using that $\bm \ra a$, we have $b \ra a$ for $b \prec \first(a)$, and $a \ra \first(a)$. Likewise, by definition of $\last(a)$ and using that $a \ra \bM$, we have $a \ra b$ for $b \succeq \last(a)$, and $b \ra a$ when $b$ is the predecessor of $\last(a)$ in $\prec$.
    
    \begin{claim}\label{claim:neigh}
        For every $a\in \Apm$,
        \[
            \dichia\big(N_{\Gpm}(a)\cup B\big) \leq \big( f(p,q,r-1)^2+1 \big)^2 \cdot \big( f(p-1,q,r) \cdot f(p,q-1,r)+1 \big).
        \]
    \end{claim}    
    \begin{proofclaim}
        Let $a \in A$, and denote $b \in B$ the vertex preceding $\last(a)$ in $\prec$. We define: 
        \begin{itemize}[itemsep=0pt]
            \item $A_1=\{a'\in \Apm : \first(a')\preceq \first(a) \prec \last(a')\}$, 
            \item $A_2=\{a'\in \Apm : \first(a') \prec \last(a) \preceq \last(a')\}$, and 
            \item $A_3=\{a'\in \Apm : \first(a) \prec \first(a')\prec\last(a')\prec\last(a)\}$. 
        \end{itemize}
        Note that $A_1$ and $A_2$ are cliques in $G$, since all intervals $I_{a'}$ for $a'\in A_1$ (resp. $a'\in A_2$) contain $\first(a)$ (resp. $b$). Therefore, using Claim~\ref{claim:clique} and Corollary~\ref{cor:cover}, we get
        \[
            \dichia \big( A_1 \cup B \big) \le f(p,q,r-1)^2
            \qquad\text{and}\qquad
            \dichia \big( A_2 \cup B \big) \le f(p,q,r-1)^2.
        \]
    
        For $a' \in A_3$, we have $\first(a) \ra a'$ as $\first(a) \prec \first(a')$, and $a'\ra b$ as $\last(a') \preceq b$, thus $A_3 \subseteq (f(a)^+ \cap b^-)$. Recall that $a\ra \first(a)$ and $b\ra a$, so $(f(a)^- \cap b^+) \ne \emptyset$. Hence, we can apply Claim \ref{claim:c4} to $\first(a)$ and $b$, which yields
        \[
            \dichia\big(A_3\cup B\big) \leq f(p-1,q,r) \cdot f(p,q-1,r).
        \]
        The claim follows from Corollary \ref{cor:cover} and the fact that $A_1 \cup A_2 \cup A_3 = N_{\Gpm}(a)$.
    \end{proofclaim}

    Let $L_0,\dots,L_h$ be the partition of $V(\Gpm)$ guaranteed by Lemma~\ref{lem:interval_graphs}. For each $0 \le i \le h$, denote $B_i = \bigcup_{a\in L_i} I_a$. By Lemma~\ref{lem:interval_graphs}, for $0 \le i \le j-2 \le h-2$, we have $\max(B_i) \prec \min(B_j)$ so $B_i$ and $B_j$ are disjoint. To make use of that, we split the layers in two sets according to their parity:
    \[
        L_{\rm even} = \bigcup_{i \rm\ even} L_i
        \qquad\text{and}\qquad
        L_{\rm odd} = \bigcup_{i \rm\ odd} L_i.
    \]
    Let $a \in L_i$ and $b \in B$. If $b \prec \min(B_i)$ then $b \prec \first(a)$ so $b \ra a$, and if $\max(B_i) \prec b$ then $\last(a) \preceq b$ so $a \ra b$. Is is easy to prove that that every directed cycle of $T[L_{\rm even}\cup B]$ is included in $T[L_i\cup B_i]$ for some $i$ (even). Hence, putting together acyclic dicolourings of $T[L_i \cup B_i]$ for each even index $i$, yields an acyclic dicolouring of $T[L_{\rm even}\cup B]$, thus
    \[
        \dichia(L_{\rm even} \cup B) \le \max_{i \text{ even}} \dichia(L_i \cup B_i).
    \]
    The same holds for $L_{\rm odd}$ and odd layers, so by Corollary~\ref{cor:cover} we get
    \[
        \dichia(A^{+-} \cup B) \le \Big( \max_{0 \le i \le h} \dichia(L_i\cup B_i)+1 \Big)^2.
    \]
    It thus suffices to bound $\dichia(L_i\cup B_i)$ for $0 \le i \le h$. By Lemma~\ref{lem:interval_graphs}, for each $0 \le i \le h$, either $|L_i| = 1$, or $i \ge 1$ and $L_i \subseteq N_{\Gpm}(u)$ for some $u \in L_{i-1}$. Hence by Claim~\ref{claim:neigh}, we have
    \[
        \dichia(L_i\cup B_i) \leq \big( f(p,q,r-1)^2+1 \big)^2 \cdot \big( f(p-1,q,r) \cdot f(p,q-1,r)+1 \big).
    \]
    This concludes the proof of the lemma.
\end{proof}

We now derive our Theorem~\ref{thm:main}, which we first restate here for convenience.

\main*

\begin{proof}
    Denote by $C_k$ the tournament $TT_k \Rightarrow (\Delta(1,1,k) \Ra TT_k)$, and let $T$ be a $C_k$-free tournament. The goal is to show that the acyclic dichromatic number of $T$ is bounded by some value $g(k)$. By Theorem~\ref{thm:heroes}, $T_k$ is a hero, so there exists an integer $h(k)$ such that $\dic (T) \le h(k)$. Hence, we can partition $T$ into $c \le h(k)$ transitive tournaments $T_1,\dots,T_c$. Remark that the subtournaments induced by two colour classes $\big\{T[T_i \cup T_j]\big\}_{i,j \in [c]}$ together cover all arcs of $T$. Thus, by Corollary~\ref{cor:cover},
    \[
        \dichia(T) \le \prod_{i,j\in[c]} \big( \dichia(T_i \cup T_j)+1 \big).
    \]    
    Therefore, it suffices to bound the acyclic dichromatic number of $T[T_i \cup T_j]$ for each $i,j \in [c]$. Fix $i,j \in [c]$. Note that $T[T_i]$ and $T[T_j]$ are acyclic induced subtournaments of $T[T_i \cup T_j]$, and after removing their edges, we are left with the bipartite tournament $T[T_i,T_j]$, so by Corollary~\ref{cor:cover}, we have $\dichia(T_i \cup T_j) \le 4 \cdot \dichia \big( T[T_i,T_j] \big)$. Denote $A = T_i$, $B = T_j$, and $T' = T[T_i, T_j]$. Let $f$ be a function that satisfies the statement of Lemma~\ref{lem:main}, and towards a contradiction, suppose that 
    \[
    \dichia(T') \ge f\big(r(k), r(k), 6k\big),
    \]
    where $r(k) = 2^{2k}+2^k+4k$.
    According to Lemma~\ref{lem:main}, either $T'$ contains a dimatching of size $r(k)$ or for every ordering of $B$ there is a vertex $a \in A$ that switches at least at $6k$ times.
    If $T'$ contains a dimatching of size $r(k)$, then $T$ contains $C_k$ by Lemma~\ref{lem:matching_to_champion}.
    
    
    Hence, assume that for every ordering of $B$, there is a vertex $a \in A$ that switches at least at $6k$ distinct indices. Denote $\prec$ the topological ordering of $B$, so that we have $b \prec b'$ if and only if $b \ra b'$. Let $a \in A$ be a vertex that switches at $6k$ different vertices $b_1 \prec \dots \prec b_{6k}$. For each $j \in [6k]$, we denote respectively $b_j^-$ and $b^+_j$ the in-neighbour and the out-neighbour of $a$ among $b_j$ and its predecessor in~$\prec$. Remark that $b_j^-$ and $b_{j+1}^-$ may not be distinct, but $b_j^-$ and $b_{j+2}^-$ are always distinct. Let 
    \[
    X = \big\{a\big\} \cup \big\{b_2^-, b_4^-, \dots, b_{2k}^-\big\} \cup \big\{b_{2k+1}^+\big\} \cup  \big\{b_{2k+2}^-, b_{2k+4}^-, \dots, b_{4k}^-\big\} \cup \big\{b_{4k+2}^+, b_{4k+4}^+, \dots, b_{6k}^+ \big\}.
    \]
    Observe that the relations below hold, and thus $X$ induces a copy of $C_k$, a contradiction.
    \[
        \big\{b_2^-, b_4^-, \dots, b_{2k}^-\big\} \Ra \Big( X \sm \big\{b_2^-, b_4^-, \dots, b_{2k}^-\big\} \Big),
    \]
    
    \[
        a \Ra b_{2k+1}^+ \Ra \big\{b_{2k+2}^-, b_{2k+4}^-, \dots, b_{4k}^-\big\} \Ra a, \text{ and}
    \]
    
    \[
        \Big( X \sm \big\{b_{4k+2}^+, b_{4k+4}^+, \dots, b_{6k}^+ \big\} \Big) \Ra \big\{b_{4k+2}^+, b_{4k+4}^+, \dots, b_{6k}^+ \big\}.
    \]
    Hence $\dichia(T') < f\big(r(k), r(k), 6k\big)$. It follows that the acyclic dichromatic number of $T$ is at most
    \[
        g(k) := \big( 4 \cdot f \big( r(k),r(k),6k \big)+1 \big)^{h(k)^2}.
    \]
    This concludes the proof of the theorem.
\end{proof}
\end{document}